\documentclass[11pt, oneside]{amsart}   
\usepackage{geometry}                		
\usepackage{graphicx}
\usepackage[normalem]{ulem}				
\usepackage{amssymb,hyperref}
\usepackage[dvipsnames]{xcolor}
\usepackage{epstopdf}
\usepackage{tikz-cd}
\usepackage[all]{xy}
\newtheorem{thm}{Theorem}[section]

\newtheorem{cor}[thm]{Corollary}
\newtheorem{defn}[thm]{Definition}
\newtheorem{ppty}{Property}

\newtheorem{lem}[thm]{Lemma}
\newtheorem{prop}[thm]{Proposition}

\theoremstyle{definition}
\newtheorem{rem}{Remark}
\newtheorem{exmpl}[thm]{Example}

\newcommand{\qh}{\widehat{\mathbb Q}}
\newcommand{\F}{\mathcal F}
\newcommand{\calD}{\mathcal D}
\newcommand{\calQ}{\mathcal Q}
\newcommand{\Z}{\mathbb Z}
\newcommand{\C}{\mathbb C}
\newcommand{\Q}{\mathbb Q}
\newcommand{\q}{\mathbb Q}
\newcommand{\whq}{\widehat{\q}}
\newcommand{\R}{\mathbb R}

\newcommand{\co}{\!:}

\title{Farey Recursive Functions}
\author{Eric Chesebro, Cory Emlen, Kenton Ke, Denise LaFontaine, Kelly McKinnie, Catherine Rigby}
\date{}							

\begin{document}
\maketitle

\section{Introduction}

A {\em second order linear recurrence relation} is an expression of the form
\[ y_{n+1} = ay_{n-1}+by_{n}\]
where the $y_j$'s are indeterminants and $a$ and $b$ are numbers.   If we take $a=b=1$, $y_0=0$, and $y_1=1$, then the sequence of numbers $\{ y_j \}_0^\infty$ which satisfies this relation is the well known sequence of {\em Fibonacci numbers}
\[ 0, 1, 1, 2, 3, 5, 8, 13, \ldots \]

Second order linear recurrence relations are prominent throughout mathematics and appear in surprising and diverse problems.  There are also many generalizations.  One possibility is to allow $a, b$, and the $y_j$'s to be polynomials.  For instance, the {\em Fibonacci polynomials} are defined by setting $y_0=0$, $y_1 = 1$, as with the first two Fibonacci numbers, and insisting that the remainder satisfy the recurrence relation
\[ y_{n+1} = y_{n-1}+xy_n.\]
The first few Fibonacci polynomials are 
\[0,1,x,x^2+1,x^3+2x, x^4+3x^2+1,\ldots\]
Evidently, when the Fibonacci polynomials are evaluated at $x=1$, the result is the Fibonacci numbers.  The Fibonacci polynomials share many interesting identities with the Fibonacci numbers (see e.g., \cite[Ch.9]{benjamin})  and just as the Fibonacci numbers solve many counting problems, so do the Fibonacci polynomials.  For instance, the coefficient of $x^k$ in $y_n$ counts the number of tilings of a $2 \times n$ grid of squares by dominoes where exactly $k$ of the dominoes are horizontal \cite[Combinatorial Theorem 12]{benjamin}.

Another famous family of polynomials which satisfy second order linear recurrence relations are the {\em Chebyshev polynomials}. Each class of Chebyshev polynomials satisfies the recurrence
\[ p_{n+1} = -p_{n-1}+2xp_n\]
and the various classes differ only in their initial conditions.  These polynomials arise naturally in the context of differential equations and trigonometric functions, but like the Fibonacci examples, they enjoy a staggering diversity of applications throughout mathematics. 

This paper investigates sets of polynomials with a more complex recursive structure.  Informally, these polynomials correspond to the vertices of the infinite graph $\calD \subset \R^2$ indicated in Figure \ref{fig:bigdaddy} and each subset of polynomials on a straight line in the graph satisfies a second order linear recurrence relation which depends on polynomials assoicated to vertices in the graph above the line.   Sets of polynomials with this structure are called {\em Farey recursive} - a precise definition is given in Section \ref{s:2}.  The graph $\calD$, referred to here as the {\em Stern-Brocot diagram}, is constructed carefully in \cite[Ch. 1]{hatcher} and is closely related to the classical {\em Farey graph}. We outline Hatcher's construction in Section \ref{sec:tree}, highlighting the parts which are especially relevant for our results in subsequent sections. For now, we mention that by projecting vertices to their $x$-coordinate, we obtain a correspondence between the extended rationals $\whq = \Q \cup \{ \infty \}$ and the vertices of $\calD$.  Thus, a set of Farey recursive polynomials can be viewed as the image of a {\em Farey recursive function} from $\whq$ to a polynomial ring.

\begin{figure}[ht]
\centering
\includegraphics[width = .85\textwidth]{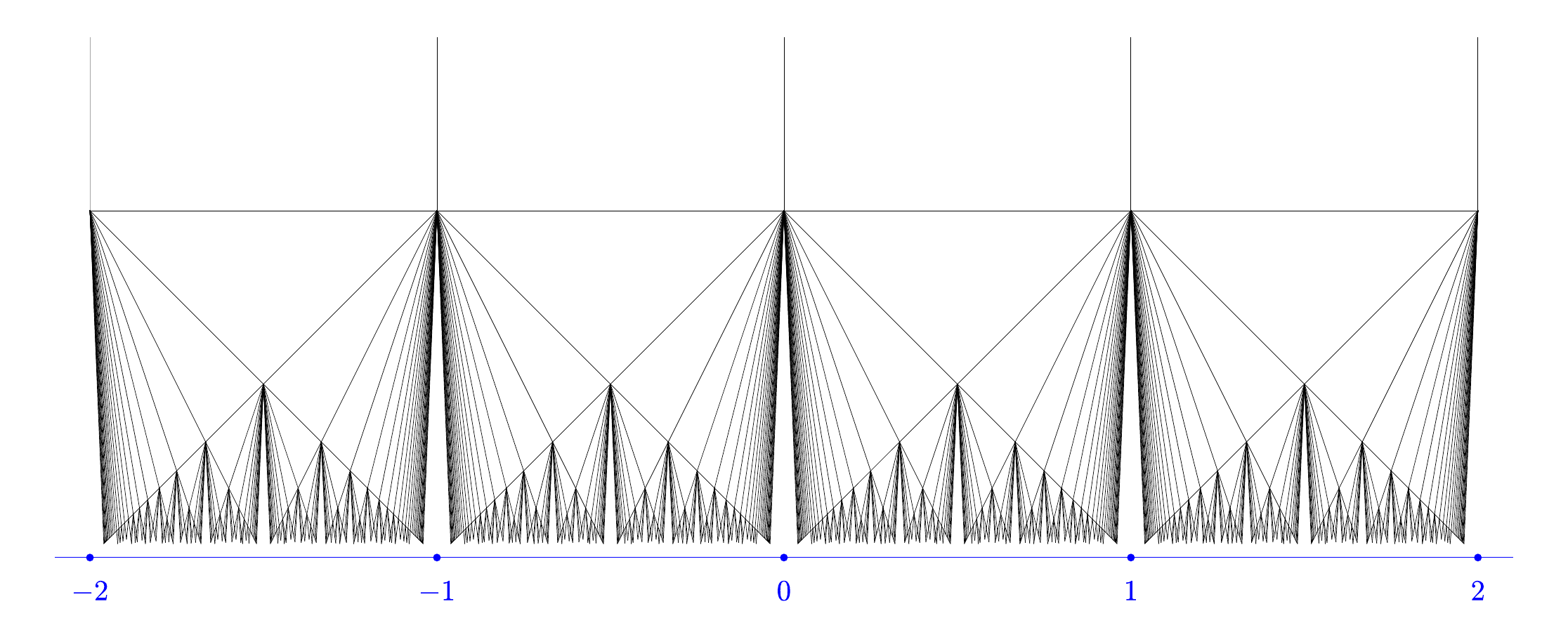}
\caption{The portion of $\mathcal{D}$ with $x$-values in $[-2,2]$}
\label{fig:bigdaddy}
\end{figure}

As seen in the above examples, a sequence with second order linear recurrence is uniquely determined by a pair of sequential terms together with the recurrence relation.  Consequently, these sequences are amenable to computer computations and proofs by induction.  One fundamental result of this paper shows that Farey recursive functions are determined concisely in a similar way.  In particular, Theorem \ref{thm:kenton} shows that Farey recursive functions on $\whq$ are defined by a triple of elements along with a bit of extra data referred to as the {\em determinant}.  Armed with this result, it is very easy to construct examples and it empowers both computers and induction as viable tools in this area.

Certain Farey recursive functions arise naturally in the geometry and topology of low-dimensional manifolds and in number theory, see \cite{Bow}, \cite{C2}, \cite{C},  \cite{MSW} and \cite{S}. This paper is motivated by these works, and its purpose is to introduce the general definition of a Farey recursive function as an interesting object in its own right, provide some interesting examples, and to establish some basic properties.  

The main theorems of this paper are Theorems \ref{thm:kenton} and \ref{thm:recursion_formula}.  The first was discussed briefly above. The second describes how Farey recursive sequences {\em wrap around triangles} in $\calD$. This property is explained in detail in Section \ref{sec:wrap}. For now, we mention that it is analogous to how  linearly recurrent sequences (e.g., the Fibonacci numbers) can often be extended into the negative direction giving a bi-infinite sequence.  


\section{The Stern Brocot Diagram}\label{sec:tree}

This section reviews Hatcher's construction of the {\em Stern-Brocot diagram} $\calD$ and describes some anatomy of $\calD$ which will be helpful in what follows. Nearly all of the content here is adapted from \cite{hatcher}.

With one exception, the vertices of the graph $\calD\subset \R^2$ correspond to the rational numbers.  Our convention is to always write rational numbers as quotients in lowest terms with non-negative denominators.  In particular, if $n \in \Z$, we write $n=n/1$.  The {\em extended rationals} $\whq$ consist of the usual rationals together with an abstract point at infinity denoted $1/0$. 

For a non-negative integer $n$, define the $n^\text{th}$ {\em Farey sequence} to be
\[ F_n = \left\{ \frac{p}{q} \, \bigg| \, q \leq n \right\}\subset \whq.\]
The elements of $F_n$ are ordered from smallest to largest.  Evidently, 
\[ F_0 \subset F_1 \subset \cdots \subset F_n \subset \cdots \qquad \text{and} \qquad \bigcup_{n \geq 0} F_n = \whq.\]
The Farey sequences are named for the geologist John Farey Sr. who, in the early 1800's, conjectured that, if $\beta \in F_{n+1}-F_n$ and $\alpha, \beta, \gamma$ are  consecutive in $F_{n+1}$, then $\beta$ is the mediant of $\alpha$ and $\gamma$.  The {\em mediant} of a pair of rational numbers $p/q$ and $r/s$ is the number $(p+r)/(q+s)$. Shortly afterwards, Cauchy supplied a proof.  Unknown to Cauchy, another mathematician, Charles Haros, had published similar results previously \cite[p.44]{HW}, \cite{CTK}.

In \cite[Ch.1\&2]{hatcher}, Hatcher gives an elementary and geometric argument for Farey's conjecture.  Hatcher's proof is outlined here, because it helps to motivate the Stern-Brocot diagram which is key to his argument and the rest of this paper.

First, a bit of notation.  Define $\mathbb{H}^2=\{ (x,y) \, | \, y>0\}$, the upper half of $\R^2$ and, for a pair of points $a,b \in \R^2$, let $[a,b]$ denote the straight line segment that connects $a$ and $b$.

In what will ultimately become the vertex set for $\calD$, there is an inductively defined collection of sets of points $P_n \subset \mathbb{H}^2$ indexed by the natural numbers.  The first set corresponds to $F_1$ and is defined as  $P_1=\{ (n,1) \, | \, n \in \Z\}$.  Notice that the points of $P_1$ have distinct $x$-coordinates, so $P_1$ is a bi-infinite ordered sequence, ordered by their integer first coordinates.  Now, suppose that $P_n \subset \mathbb{H}^2$ is a bi-infinite sequence of rational points with distinct $x$-coordinates (whose denominators are at most $n$), ordered by first coordinates.  A pair of consecutive points in $P_n$ constitute the upper two corners of a quadrilateral whose bottom lies on the $x$-axis and whose sides are  vertical.  The intersection of the diagonals of the quadrilateral lies in its interior, so this intersection point lies in $\mathbb{H}^2$.  Also, its $x$-coordinate is distinct from the $x$-coordinates of every point in $P_n$.   $P_{n+1}$ is defined to be the union of all such intersection points (whose denominators are at most $n+1$) together with the points of $P_n$.  Since the $x$-coordinates of the points in $P_{n+1}$ are distinct, $P_{n+1}$ is a bi-infinite ordered sequence, ordered by first coordinates.  Points of $P_n$ can be and are often identified with their first coordinates.

It is now possible to define the {\em Stern-Brocot diagram} $\mathcal D$. Let $\calD_0$ be the union, over all $n \in \mathbb{N}$, of all line segments $[a,b]$ where $a$ and $b$ are adjacent in $P_n$.  Define
\[ \calD = \left\{ (n,t) \, | \, n \in \Z \text{ and } t\geq 1 \right\} \cup \calD_0.\]
The {\em vertices of $\calD$} are the points 
\[\bigcup_{n\in \mathbb N} P_n \cup \left\{\frac10\right\}.\] 
The exceptional point $\frac10$ is called the {\em vertex at infinity}. It arises by compactifying the non-compact ends of the vertical rays $\{ (n,t) \, | \, t\geq1\}$ with a single point.  After setting $P_0 = \{ 1/0 \}$, the full vertex set of $\calD$ is $\cup_0^\infty P_n$.  Some pieces of $\calD$ are pictured in Figures \ref{fig:bigdaddy} and \ref{fig:shorty}.  Very nice pictures of this construction can be found in Chapter 1 of \cite{hatcher}.  

Suppose that $n \geq 1$ and that $b \in P_{n+1}-P_n$.  The adjacent terms $a,c \in P_{n+1}$ (so that $a<b<c$ and no point in $P_{n+1}$ lies between $a$ and $b$ or $b$ and $c$) are called $b$'s {\em parents}.  The point $a$ is called the {\em left parent} for $b$ and $c$ is the {\em right parent} for $b$.

\begin{figure}[ht]
\centering
\includegraphics[width = .5\textwidth]{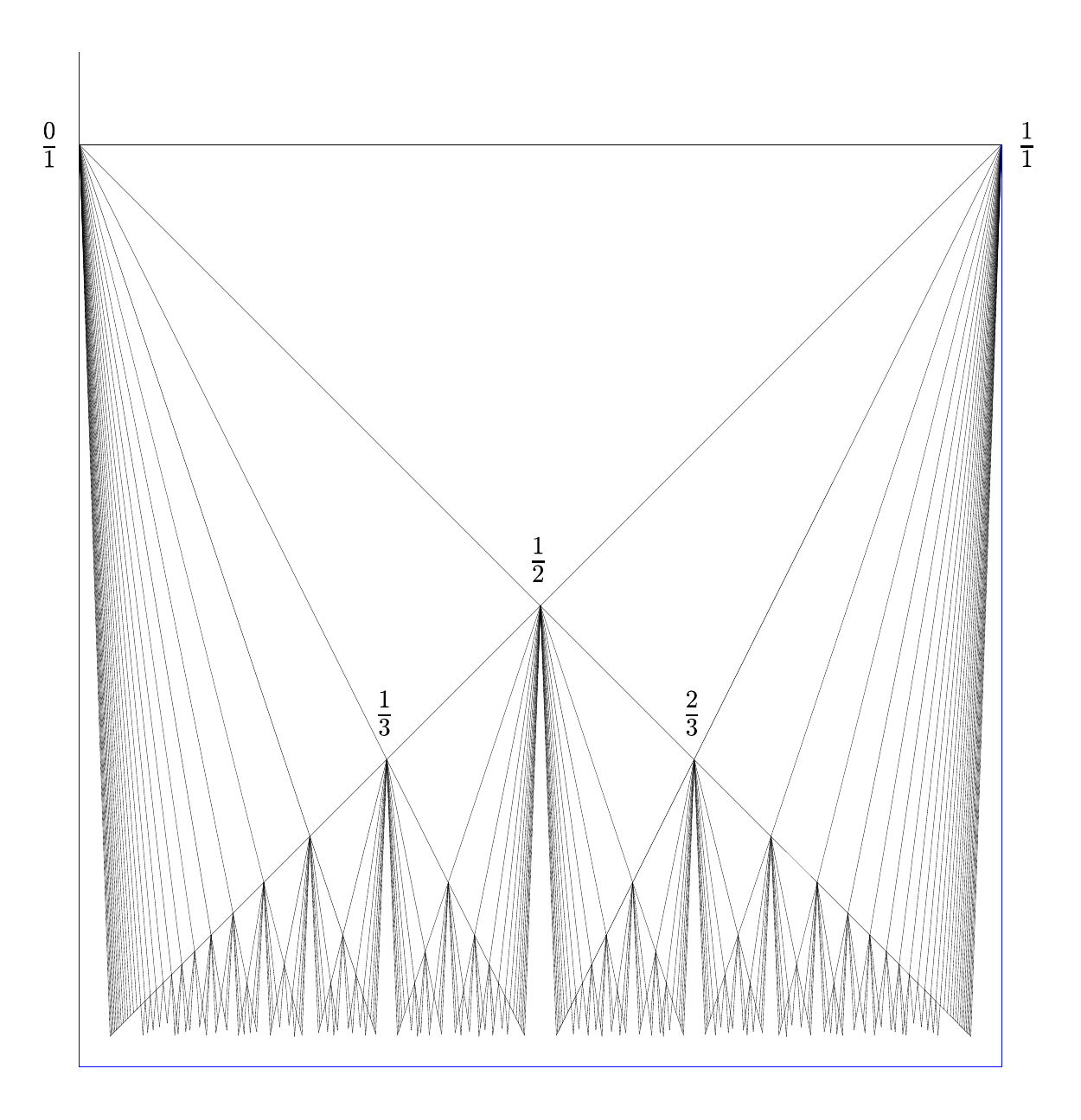}
\caption{The portion of $\calD$ over $[0,1]$.}
\label{fig:shorty}
\end{figure}

The next lemma follows easily using analytic geometry and is proven in \cite[p.20]{hatcher}
\begin{lem}\label{lem:quad}
If $p/q$ and $r/s$ are distinct rational numbers, then the diagonals of the quadrilateral with vertex set
\[\left\{(p/q,1/q), (p/q,0), (r/s,1/s), (r/s,0)\right\}\]
intersect at the point $\left(  \frac{p+r}{q+s}, \frac{1}{q+s}\right)$.
\end{lem}

Lemma \ref{lem:quad} shows that if the points of $P_n$ all have the form $(p/q,1/q)$ for some $p,q\in \Z$ then the new points in $P_{n+1}$, which are the intersection points of the diagonals in the lemma, also take this form. Since the  points in $P_1$ look like $(n/1,1/1)$ for $n \in \Z$, induction proves the first statement in the next proposition. The second statement follows directly from Lemma \ref{lem:quad}.

\begin{prop}\label{prop:mediant}
For all $n$, every point of $P_n$ is of the form $(p/q,1/q)$ where $p$ and $q$ are integers and $q$ is positive.  Moreover, if $p/q \in P_n$ then $p/q$ is the mediant of its parents.
\end{prop}

A pair of fractions $p/q$ and $r/s$ are called a {\em Farey pair} if the determinant, $ps-qr$, of the matrix $\left[\begin{smallmatrix} p & r \\ q & s \end{smallmatrix}\right]$ is $\pm 1$.  Since switching the columns of a matrix change the sign of its determinant, this notion is well-defined regardless of which fraction comes first.  Consecutive integers (i.e., the $x$-coordinates of adjacent points in $P_1$) serve as a first example of Farey pairs.  As the properties below show, mediants of Farey pairs produce more Farey pairs and, as such, the mediant of a Farey pair has special notation. If $\alpha$, $\beta$ is a Farey pair, their median is called a {\em Farey sum} and is denoted $\alpha \oplus \beta$. A {\em Farey triple} is a collection of three rational numbers, any pair of which is a Farey pair.

\begin{ppty} \label{ppty:gcd}
 If $p/q$ and $r/s$ are a Farey pair then $\gcd(p+r, q+s)=1$.
\end{ppty}

\begin{ppty} \label{ppty:triple}
 If $\alpha$ and $\beta$ are a Farey pair then $\alpha$, $\beta$, and $\alpha \oplus \beta$ form a Farey triple.  
\end{ppty}

\smallskip \noindent
Property \ref{ppty:gcd} follows because, if $p+r=km$ and $q+s=kn$, then $ps-qr=k(ms-nr)$.  A quick calculation shows that the determinants of $\left[ \begin{smallmatrix} p & r \\ q & s \end{smallmatrix} \right]$ and $\left[ \begin{smallmatrix} p & p+r \\ q & q+s \end{smallmatrix} \right]$ are equal.  This establishes Property \ref{ppty:triple}.

Together with Proposition \ref{prop:mediant}, these properties imply that every pair of vertices in $\calD$ which are connected by an edge make a Farey pair.  By Property \ref{ppty:gcd}, when these coordinates are computed by taking Farey sums, the resulting fractions will never need to be simplified by canceling common factors.  Lastly, by Property \ref{ppty:triple}, if $\alpha$ and $\beta$ are a Farey pair, then Farey sums can be taken repeatedly with $\alpha$ as $(((\beta \oplus \alpha) \oplus \alpha) \oplus \cdots \oplus \alpha)$.  This repeated sum is denoted as $ \beta \oplus^k \alpha$, where $k$ is the number of $\alpha$ summands, and will play a pivotal role in our definition of Farey Recursive Functions in Section \ref{s:2}.  Observe that, if $w/x$ and $y/z$ are a Farey pair, then
\begin{align} \label{eq:repeat} \frac{w}{x} \oplus^k \frac{y}{z}  &= \frac{w+ky}{x+kz}\end{align}

To complete a proof of Farey's conjecture, it suffices to show that $F_n$ is precisely the set of $x$-coordinates of $P_n$.  Hatcher does this by utilizing a beautiful connection between mediants, continued fraction expansions, and matrix multiplication.  This is discussed next.

Suppose $p/q$ is a rational number.  The Euclidean algorithm can be used to find a {\em continued fraction expansion} for $p/q$.  This expresses $p/q$ as
\[
\frac{p}{q}=a_0+\cfrac{1}{a_1+\cfrac{1}{\begin{array}{ccc} a_2 + & &  \\ & \ddots & \\ & & \frac{1}{a_n} \end{array}}} \]
where $a_0 \in \Z$ and $a_i \in \mathbb{N}$ for $i\geq 0$.  The expression on the right hand side is often abbreviated as $[a_0; a_1, \ldots , a_n]$. 

Let $t_i$ and $u_i$ be the numerator and denominator of $ [0; a_{i}, \ldots, a_n]$ when expressed in lowest terms.  Then
\begin{equation*}\label{eq:1} 
\frac{t_{i-1}}{u_{i-1}} = \left( a_{i-1} + \frac{t_i}{u_i} \right)^{-1}= \frac{u_i}{ t_i + a_{i-1}u_i} 
\end{equation*}
and
\begin{equation}\label{eq:2}
 \begin{bmatrix} 0 & 1 \\ 1 & a_{i-1} \end{bmatrix} \begin{bmatrix} t_i \\ u_i \end{bmatrix} =\begin{bmatrix} u_i \\ t_i + a_{i-1}u_i \end{bmatrix} = \begin{bmatrix}t_{i-1}\\u_{i-1}\end{bmatrix}.
 \end{equation}
Multiply the matrices on the left hand side of  (\ref{eq:prod}) together starting on the right side of the product. Repeated use of equation (\ref{eq:2}) shows
\begin{align} \label{eq:prod} \begin{bmatrix} 1 & a_0 \\ 0 & 1 \end{bmatrix} \begin{bmatrix} 0 & 1 \\ 1 & a_1 \end{bmatrix}\begin{bmatrix} 0 & 1 \\ 1 & a_2 \end{bmatrix} \cdots  \begin{bmatrix} 0 & 1 \\ 1 & a_n \end{bmatrix} \begin{bmatrix} 0  \\ 1  \end{bmatrix} =\begin{bmatrix} 1 & a_0 \\ 0 & 1 \end{bmatrix}\begin{bmatrix}t_1\\u_1\end{bmatrix} = \begin{bmatrix}t_1+u_1a_0\\u_1\end{bmatrix}.\end{align}
So, the quotient of the entries in the column vector (\ref{eq:prod}) is $p/q$.

On the other hand, we can start the multiplication in (\ref{eq:prod}) from the left.  Note that for any Farey pair $x/y$ and $z/w$,
\[ \begin{bmatrix} x & z \\ y & w \end{bmatrix}   \begin{bmatrix} 0 & 1 \\ 1 & a_i  \end{bmatrix} =  \begin{bmatrix} z& x+a_iz \\ w & y+a_iw \end{bmatrix}.\]
The right hand side matrix corresponds to the Farey pair $z/w$ and $x/y\oplus^{a_i}z/w$ and shows that $p/q$, the quotient of the column entries in (\ref{eq:prod}), is equal to a pattern of repeated Farey sums of Farey pairs, each taken $a_i$ times, $0\leq i \leq n$, beginning with the Farey pair $1/0$ and $a_0/1$ (see \cite[Th. 2.1]{hatcher} for more detailed explanation). The following lemma is now needed.
\begin{lem}\label{lem:con}Let $x/y, z/w \in \Q$ be  consecutive elements of $P_n$ where $n=\max(y,w)$. Then, for every $k\geq 0$, $x/y \oplus^k z/w$ is in $P_{y+kw}$ and is consecutive with $z/w$.
\end{lem}
\begin{proof} This is true by induction since $x/y\oplus z/w \in P_{y+w}$ and $x/y\oplus z/w$ is consecutive with $z/w$ in $P_{y+w}$.
\end{proof}

Lemma \ref{lem:con} says that if we have a matrix $\begin{bmatrix}x&z\\y&w\end{bmatrix}$ with $x/y$ and $z/w$ a consecutive Farey pair in $P_{\max(y,w)}$, then 
\[\begin{bmatrix}x&z\\y&w\end{bmatrix}\begin{bmatrix}0&1\\1&a_i\end{bmatrix}\]
has column ratios which are a consecutive Farey pair in $P_{y+a_iw}$. Since the first two matrices of (\ref{eq:prod}) can be rearranged as
\[\begin{bmatrix} 1 & a_0 \\ 0 & 1 \end{bmatrix} \begin{bmatrix} 0 & 1 \\ 1 & a_1 \end{bmatrix} = \begin{bmatrix} a_0 & 1+a_0a_1 \\ 1&a_1 \end{bmatrix} = \begin{bmatrix}a_0+1&a_0\\1&1\end{bmatrix}\begin{bmatrix}0&1\\1&a_1-1\end{bmatrix}\]
and $a_0/1$, $(a_0+1)/1$ are a consecutive Farey pair in $P_1$, we see that $\begin{bmatrix}p\\q\end{bmatrix}$ is the column vector in (\ref{eq:prod}) (that is, no cancelling happens in the ratio) and is an element of $P_q$. This establishes Farey's conjecture and shows that $\Q = \cup_nP_n$, the vertices of $\calD$.

Recall from above that every edge of $\calD$ connects a Farey pair.  In the next section we prove the converse.  This also provides an opportunity to introduce important definitions and geometric properties of $\calD$ which are used in the definition of Farey recursive functions.

\subsection{Boundary Sequences in $\calD$}

Recall that we identify the elements of $\Q$ with the vertices in $\mathcal D$ by $p/q \leftrightarrow (p/q,1/q)$ and  use both notations interchangeably throughout. Given $\alpha \in \whq$ define the {\bf boundary of $\alpha$} to be the set
\[\partial(\alpha) = \{\beta \in \whq \,|\, \alpha \,\textrm{and}\, \beta \textrm{ are a Farey pair}\}.\]
The geometry of these boundaries vary depending on $\alpha$ and fall into three cases; $\alpha = 1/0$, $\alpha = n/1$ and $\alpha \in \Q-\Z$. The following propositions establish basic properties about boundaries in the three cases.  This is important because the definition of Farey recursion in Section \ref{s:2} requires second order linear recursion on boundaries.

\begin{prop}\label{prop:bdry0} For $n \in \Z$,
\[ \partial(n/1) = \left\{ \frac{kn- 1}{k} \, \Big| \, k \in \Z_{> 0} \right\} \cup \left\{ \frac{kn+ 1}{k} \, \Big| \, k \in \Z_{\geq 0} \right\}\]
and $\partial(1/0)$ constists of the integer points $P_1$.
\end{prop}

\begin{proof} Both statements follow directly from the definition of a Farey pair.
\end{proof}

\begin{prop} \label{prop:bdry}Let $p/q \in \Q-\Z$ and let $\gamma_L$ and $\gamma_R$ denote the left and right parents of $p/q$. Then
\[\partial(p/q) = \{\gamma_L\oplus^kp/q\,|\,k \geq 0\}\cup \{\gamma_R\oplus^kp/q\,|\,k \geq 0\}\]
\end{prop}

\begin{proof} Let $r/s$ and $t/u$ be the left and right parents of $p/q$.  Because $r/s$, $p/q$, and $t/u$ are connected pairwise by edges in $\calD$, they form a Farey triple.  Moreover, since $q>1$ we have  $0<s,u<q$.  Also, $r/s<p/q<t/u$ implies
\begin{equation}\label{eq:4} ps-rq=1 \quad \text{and} \quad pu-tq=-1.
\end{equation}
That is $(x,y) = (r,s)$ is a solution to the linear equation $py-xq = 1$ and $(x,y) = (t,u)$ is a solution to $py-xq = -1$. By \cite[Lem 2.4]{hatcher}, all solutions to $py-xq = 1$ are of the form $(r+kp,s+kq)$ for some $k \in \Z$ and all solutions to $py-xq = -1$ are of the form $(t+kp,u+kq)$ for some $k \in \Z$. Take $x/y \in \partial(p/q)$. Then $(x,y)$ is a solution to one of (\ref{eq:4}) with $y>0$. In particular, either $(x,y) = (r+kp,s+kq)$ or $(x,y) = (t+kp,u+kq)$. Since $0<s,u<q$, $s+kq>0$ iff $k\geq 0$ and $u+kq > 0$ iff $k \geq 0$. In particular, $x/y = r/s\oplus^kp/q$ or $x/y = t/u\oplus^kp/q$ for some $k\geq$, showing that $x/y$ is in one of the two subsets in the proposition.
\end{proof}

For $\alpha = p/q \in \Q$, denote the two subsets making up $\partial(\alpha)$ in Propositions \ref{prop:bdry}/\ref{prop:bdry0} by $\partial_L(\alpha)$ and $\partial_R(\alpha)$, called the {\em left } (resp. \emph{right}) \emph{boundary sequence} for $\alpha$. If $\alpha=n/1 \in \Z$ then  Note that
\[ \partial_L(\alpha)=\left\{ \beta  \in \partial(\alpha) \, | \, \beta < \alpha \right\} \,\,\textrm{ and }\,\, \partial_R(\alpha)=\left\{ \beta  \in \partial(\alpha) \, | \, \beta > \alpha \right\}. \]
Moreover, by Proposition \ref{prop:bdry}, $\gamma_L$ is the element of $\partial_L(\alpha)$ with smallest denominator (similarly for $\gamma_R$) and hence lies vertically highest in $\mathcal D$ of all the elements in $\partial_L(\alpha)$. The following corollary shows that all of the elements of $\partial(\alpha)$ lie on a euclidean triangle in $\mathbb H^2$ containing $\alpha$ in the interior and with corners $\gamma_L$, $\gamma_R$ and $(p/q,0)$. This triangle is denoted $\Delta(\alpha)$ and called the {\em boundary triangle of $\alpha$}. We will refer to the parents $\gamma_L$ and $\gamma_R$ of $\alpha$ as the {\em left and right corners} of $\alpha$ (or as the corners of $\Delta(\alpha)$). 

\begin{cor}
Let $\alpha = p/q \in \Q-\Z$. The elements of $\partial_L(\alpha)$ lie on the line through $\gamma_L$ and $(\alpha,0)$ which has slope $-q$. The elements of $\partial_R(\alpha)$ lie on the line through $\gamma_R$ and $(\alpha,0)$ which has slope $q$. 
\end{cor}

\begin{proof} Let $\gamma_L(\alpha) = r/s$. By Proposition \ref{prop:bdry} , an element of $\partial_L(\alpha)$ is of the form $(\frac{r+kp}{s+kq},\frac1{s+kq})$ and thus is on the specified line. The second statement follows similarly. 
\end{proof}

The boundary sequences $\partial(1/0)$ and $\partial(n/1)$ do not lie on a euclidean triangle in $\mathbb H^2$, but we still refer to their boundary ``triangles''. As with the case above, these triangles are the union of the line segments in $\calD$ connecting the vertices in their boundary sequences. In particular, $\Delta(\frac10)$ is the line $y=1$ (and has no ``corners'') and $\Delta(n/1)$ is the union of the line segments from $(\frac{n-1}1,1)$ to $(\frac n1,0)$ to $(\frac{n+1}1,1)$ with the line segments $\{(n-1,t)\,|\,t\geq 1\}$ and $\{(n+1,1\,|\,t\geq 1\}$ (one could say this triangle has one ``corner''; $\frac10$). To illustrate this idea the boundary triangles $\Delta(1/2),\,\Delta(7/5)$ and $\Delta(2/1)$ are pictured in Figure \ref{fig:triangles}
 
\begin{figure}[ht] 
\centering
\includegraphics[width =5in]{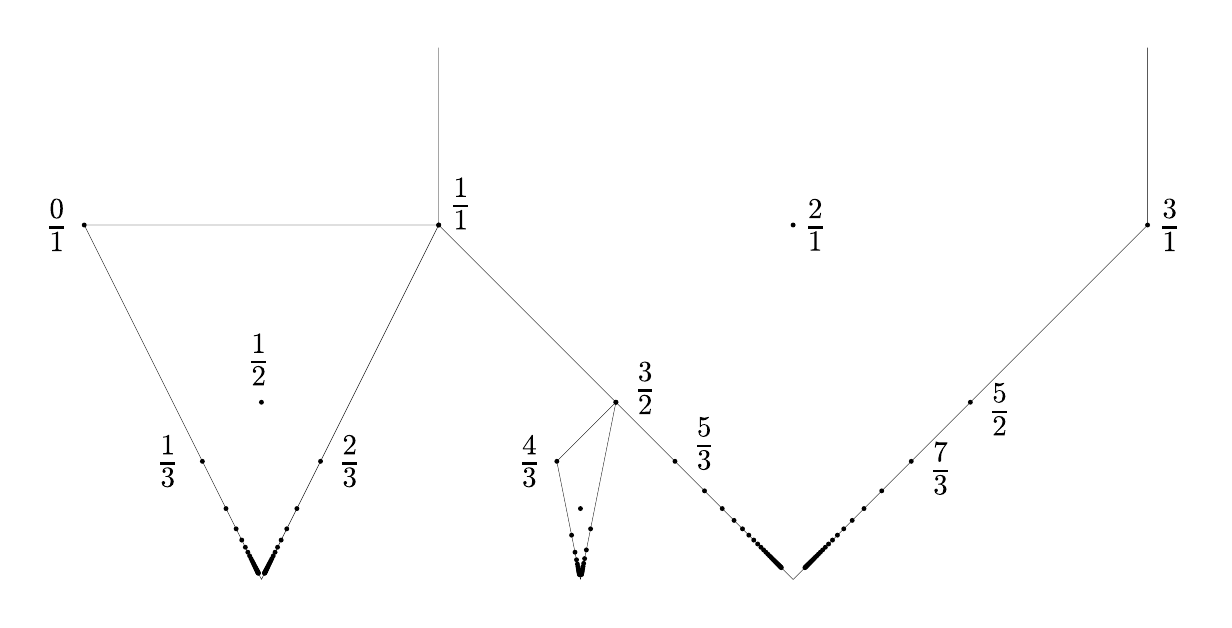}
\caption{Triangles $\Delta(1/2)$, $\Delta(7/5)$, and $\Delta(2)$.}
\label{fig:triangles}
\end{figure}

Notice that, if $\alpha \in \Q$, then $\partial_L(\alpha)$ and $\partial_R(\alpha)$  are naturally ordered as infinite sequences and each term is obtained from its predecessor by Farey summing with the center $\alpha$.

\begin{prop}\label{prop:edges}
If $\alpha \in \whq$ then every term $\gamma$ in  $\partial(\alpha)$ is connected to $\alpha$ and to $\gamma\oplus \alpha$ by an edge in $\calD$.
\end{prop}
\begin{proof}  We first dispense with the case of $\alpha$ or $\gamma$ is $1/0$. Suppose first that $\alpha =1/0$ and $\gamma \in \partial(1/0)$.  Then $\gamma=n/1$ and, by definition of $\calD$, $[\alpha, \gamma]$ is an edge in $\calD$.  Since $\gamma\oplus\alpha = (n+1)/1$ is consecutive with $\gamma = n/1$ in $P_1$, $[\gamma,\gamma\oplus\alpha]$ is an edge in $\calD$.

Similarly, if $\gamma = 1/0 \in \partial(\alpha)$ then $\alpha=n/1 \in \Z$ and both $\alpha = n/1$ and $\alpha \oplus \gamma =  (n+1)/1$ are connected to $\gamma = 1/0$ in $\calD$.

Now take $\alpha=p/q \in \Q$ and $\gamma=r/s \in \partial_R(\alpha)$ (the $\partial_L(\alpha)$ case follows similarly) with $s \ne 0$. Then by Propositions \ref{prop:bdry} and \ref{prop:bdry0}, $\gamma = \gamma_R\oplus^k\alpha$ for some $k\geq 0$.  By Lemma \ref{lem:con} $\gamma = \gamma_R\oplus^k\alpha$ is consecutive with $\alpha$ in $P_{s+kq}$ and also, hence, with $\gamma\oplus \alpha = \gamma_R \oplus^{k+1}\alpha$ in $P_{s+(k+1)q}$.
\end{proof}

Aside from providing careful descriptions of the triangles $\Delta(\alpha)$, the work above provides the last fact needed to establish the following proposition.

\begin{prop} \label{prop:converse}
The $x$-coordinates of the finite vertices of $\calD$ are precisely the set of rational numbers.  A pair of vertices are connected by an edge in $\calD$ if and only if their $x$-coordinates make a Farey pair.   
\end{prop}
\proof
The first sentence of the proposition follows from Farey's conjecture, which has been proven.  It has also been established that every edge of $\calD$ connects a Farey pair.  So, to complete the proof, assume that $p/q$ and $r/s$ make a Farey pair.    Then $r/s\in\partial(p/q)$.  Therefore, by Proposition \ref{prop:edges}, $p/q$ and $r/s$ are connected by an edge in $\calD$.
\endproof

\begin{cor}\label{cor:unique} Let $\alpha \in \Q$. There exists a unique Farey pair $\gamma_L,\gamma_R \in \whq$ with $\alpha= \gamma_L\oplus\gamma_R$. 
\end{cor}
\begin{proof} By Proposition \ref{prop:converse} $\alpha \in P_n$ for some $n\geq 1$. Let $n$ be the smallest integer for which $\alpha \in P_n$. If $n>1$ then, by definition, $\alpha$ has a unique set of parents, $\gamma_L$ and $\gamma_R$ which, by the definition of $\calD$, are connected by an edge in $\calD$. If $\alpha = n/1 \in P_1$ then $\alpha = 1/0\oplus (n-1)/1$ and the pair $1/0,(n-1)/1$ are connected by an edge in $\calD$. That is, a solution to the equation exists which are connected by an edge in $\calD$. To see that the solution is unique, consider any other Farey pair $\gamma_L',\gamma_R'$ with $\alpha = \gamma_L'\oplus\gamma'_R$. By Proposition \ref{prop:converse} $\gamma'_L$ and $\gamma'_R$ are connected by an edge in $\calD$ and hence are $\alpha$'s parents.
\end{proof}


\section{Definition of Farey Recursive Functions and Examples}\label{s:2}

As mentioned in the introduction, there many famous sequences which exhibit a second order linear recurrence. The Fibonacci relation, satisyfing $x_{n+1}=x_{n-1}+x_{n}$ can be expressed by the matrix equation
\[ \begin{bmatrix} 0 & 1 \\ 1 & 1 \end{bmatrix} \begin{bmatrix} x_{n-1} \\ x_n \end{bmatrix} = \begin{bmatrix} x_n \\ x_{n+1} \end{bmatrix}.\]
Similarly, a general second order linear recurrence $x_{n+1}=ax_{n-1}+bx_{n}$ corresponds to its {\em recursion matrix} $\left[ \begin{smallmatrix} 0 & 1 \\ a & b \end{smallmatrix} \right]$. In particular, the recursion matrices for the Fibonacci polynomials and the Chebyshev polynomials are $\left[ \begin{smallmatrix} 0 & 1 \\ 1 & x \end{smallmatrix} \right]$ and $\left[ \begin{smallmatrix} 0 & 1 \\-1 & 2x \end{smallmatrix} \right]$, respectively.

The goal here is to define a function on $\whq$ or equivalently on the vertex set of $\calD$.  The values of these function are often polynomials, but in general we only need the image to lie in a ring. The function has second order linear recursion on the boundary sequences for all $\alpha$ and the recursion matrix for the sides of $\Delta(\alpha)$ only depend on $\alpha$.  

\begin{defn} Let $R$ be a commutative ring and suppose $d_1$ and $d_2$ are functions from $\whq$ to $R$.  A function  $\F\co\widehat{\Q}\rightarrow R$ is a $(d_1,d_2)$-{\em Farey Recursive Function (FRF)} if, whenever $\alpha \in \whq$ and $\gamma \in \partial(\alpha)$  
\begin{equation} 
\F\left(\gamma\oplus^2\alpha\right)=-d_1(\alpha)\F\left(\gamma \right)+d_2(\alpha)\F\left(\gamma\oplus \alpha\right).
\label{eq:FRFdef}
\end{equation}
\end{defn}

In other words, the boundary sequences down the sides of $\Delta(\alpha)$ are linearly recursive with recursion matrix $\left[ \begin{smallmatrix} 0 & 1 \\ -d_1(\alpha) & d_2(\alpha) \end{smallmatrix} \right]$.

Every example discussed in this paper satisfies the additional property $d_2=\F$.  When this is true, set $d=d_1$ and $\F$ is referred to as an FRF with {\em determinant} $d$.  The recursion matrix for such a function, down the sides of a triangle $\Delta(\alpha)$ is $\left[ \begin{smallmatrix} 0 & 1 \\ -d(\alpha) & \F(\alpha) \end{smallmatrix} \right]$.

At this point it is natural to wonder whether Farey Recursive Functions exist. In this section we give two simple examples of FRFs with determinant $d$ where $d$ is the zero function. After the proof of Theorem \ref{thm:kenton} we will be able to give many more interesting examples.

\begin{exmpl}\label{ex:mult}
Let $R$ be a commutative ring and $\F\co\whq \to R$ a function that satisfies $\F(\gamma \oplus \alpha) = \F(\alpha)\F(\gamma)$ for every Farey pair $\alpha,\gamma \in \whq$. We refer to such functions as {\em multiplicative}. Set $d \co \whq \to R$ to be the zero function. Then
\begin{eqnarray*}
\F(\gamma\oplus^2 \alpha) &=& 0+\F(\alpha)\F(\gamma\oplus\alpha)\\
&=&-d(\alpha)\F(\gamma)+\F(\alpha)\F(\gamma\oplus\alpha).
\end{eqnarray*}
and hence $\F$ is a FRF with zero determinant.
\end{exmpl}

\begin{exmpl} \label{ex:dQ}
Suppose that $a$ and $b$ are elements of a commutative ring $R$.  The function $m_{a,b} \co \whq \to R$ given by $p/q \mapsto a^pb^q$ is multiplicative in the sense of Example \ref{ex:mult}.  Therefore $m_{a,b}$ is an FRF with zero determinant.  

An important FRF of this type is $d_\calQ \co \whq \to \Z[x]$ defined by $d_\calQ (p/q) = (-1)^p x^q$.  This FRF becomes important in later examples and applications.
\end{exmpl}

\begin{rem} \label{rem:group} If $d=0$ then the ring structure of $R$ is not used and it is possible to take $R$ to be a group.  If the group is written additively, then the {\em multiplicative} condition from Example \ref{ex:mult} becomes an additive condition $\F(\gamma \oplus \alpha) = \F(\gamma)+\F(\alpha)$.
\end{rem}

\begin{exmpl}
In  \cite{series}, Series defines the mod 2 equivalence of an element of $\widehat{\Q}$.  She uses this definition to classify primitive elements of the free group of rank 2.  For an integer $p$, let $\bar{p}$ be its class in $\Z_2$, the integers modulo 2.  The map $\phi \co \Z \times \Z \to \Z_2 \times \Z_2$ given by $\phi(p,q) = (\bar{p},\bar{q})$ is a group homomorphism and a Farey sum $p/q \oplus r/s$ corresponds to the group operation $(p,q)+(r,s)$ in $\Z \times \Z$.   Hence, Remark \ref{rem:group} applies in this situation and the function $B \co \whq \to \Z/2\Z \times \Z/2\Z$ defined by $p/q \mapsto (\bar{p}, \bar{q})$ is an FRF with zero determinant. 
\end{exmpl}


\section{ Farey triples determine Farey recursive functions}\label{s:3}

It is natural to wonder how common FRFs are.  One worry is that, because there are multiple paths down the edges of $\calD$ to a given vertex, the recursion condition may be too much to ask for.  Interestingly, the existence of unique paths to vertices is not needed to successfully define an FRF inductively from a set of initial values.  In fact, the following uniqueness statement suffices.  Similarly to the situation with linear recurrences, this will provide an easy way to define an FRF from given functions $d_1$ and $d_2$ and a triple of initial conditions.

\begin{lem}\label{unique_triple}
For every $p/q \in \Q$ there exists a unique Farey pair $x/y,\,z/w \in \qh$ such that
\begin{equation}\frac pq = \frac xy \oplus^2 \frac zw. \label{e0}\end{equation}
\end{lem}

\begin{proof}

When $q=1$ the desired decomposition in (\ref{e0}) forces
\[\frac p1 = \frac{p-2}1\oplus\frac10\oplus\frac10.\]
Since $(p-2)/1$ and $1/0$ make a Farey pair, this is a valid and unique decomposition. 

Next, consider the case when $q=2$. There are two possibilities: $y=2$, or $w=1$. 
If $y=2$, then $w=0$ which forces $z=1$. This is not a possible decompositon since $(p-2)/2$ and $1/0$ are not a Farey pair. Hence we must have $w=1$ which forces the decomposition
\[\frac p2 = \frac10\oplus \frac{(p-1)/2}1\oplus \frac{(p-1)/2}1.\]
Note that $q=2$ means $p$ is odd so $(p-1)/2 \in \Z$ showing the decomposition is valid and unique. 

Assume now $q\geq 3$.  If $\gamma_L = r/s$ and $\gamma_R=r'/s'$ are the parents of $p/q$, then Corollary \ref{cor:unique} shows that they are the unique Farey pair of rational numbers so that $p=r+r'$ and $q=s+s'$.  Suppose first that $s=s'$. Because $r/s$ and $r'/s$ make a Farey pair, $rs-r's=s(r-r')=\pm 1$. Thus $s=1$ implying $q=2$ which is false.  Therefore, without loss of generality, we may assume $s<s'$. Notice that $(r'-r)/(s'-s)$ and $r/s$ are a Farey pair and moreover, 
\[\frac{p}{q}=\frac{r'-r}{s'-s}\oplus \frac{r}{s}\oplus \frac{r}{s}\] 
giving the existence of the decomposition.
To see uniqueness, assume there is another Farey pair $x/y$, $z/w$ with 
\[\frac{p}{q}=\frac{r'-r}{s'-s}\oplus \frac{r}{s}\oplus \frac{r}{s}=\frac{x}{y}\oplus \frac{z}{w}\oplus \frac{z}{w}.\]
Then $x/y\oplus z/w$ and $z/w$ form a Farey pair whose sum is $p/q$ showing that it is the same as the pair $r/s$ and $r'/s'$. Since $y+w>w$ ($y$ cannot be zero since $q\geq3$), this forces $z/w = r/s$ and $(x+z)/(y+w)=r'/s'$ resulting in $x/y = (r'-r)/(s'-s)$.
\end{proof}

\begin{rem} \label{rem:size}
Suppose $p/q \in \Q$ and $q\geq 2$.  Let $x/y,z/w \in \Q$ as given by Lemma \ref{unique_triple}.  Since $z/w$ must be a Farey partner for $p/q$ and the only Farey partners for $1/0$ are integers, it must be true that $w>0$.  So, because $q=y+2w$, $q$ is larger than $y$, $w$ and $y+w$.
\end{rem}

The next two theorems show that Farey recursive functions (FRFs) are easy to construct using initial values as with the Fibonacci numbers. Remark \ref{rem:size} makes it possible to define FRFs inductively.   Indeed, Theorem \ref{thm:kenton} shows that,  if $d_1$ and $d_2$ are arbitrary functions from $\whq$ to a commutative ring $R$, $d_1(1/0)$ is invertible in $R$ and $a,b,c \in R$, then there is a $(d_1,d_2)$-FRF $\F$ on $\whq$ which maps the triple $(1/0, 1/1,0/1)$ to the triple $(a,b,c)$.  In Theorem \ref{thm:shortQ} the hypothesis that $d_1(1/0)$ is invertible in $R$, is removed. The resulting function is still a $(d_1,d_2)$-FRF but it is only defined  on a smaller portion of $\whq$.  

\begin{thm} \label{thm:kenton} Suppose that $a$, $b$, and $c$ are elements of a commutative ring $R$ and $d_1$ and $d_2$ are functions from $\whq$ to  $R$.  If $d_1(1/0)$ is invertible in $R$ then there is a unique $(d_1,d_2)$-Farey recursive function $\F\co\whq \to R$ with $\F\left(0/1 \right)=a$,  $\F\left(1/0 \right)=b$, and  $\F\left(1 \right)=c$.
\label{thm:main}
\end{thm}

\begin{proof}
$\F(p/q)$ will be defined inductively on the denominator $q$, that is, on $n$ in the Farey sequences $F_n$, $n\geq 1$.  First, define $\F(0) = a$, $\F(0/1) = b$ and $\F(1) = c$ and let $M= \left[ \begin{smallmatrix}0&1\\-d_1(1/0)&d_2(1/0) \end{smallmatrix} \right]$.
Since $d_1(1/0)$ is invertible in $R$, $M$ is invertible in $M_2(R)$. For $q=1$ we saw in the proof of (\ref{unique_triple}) that a decomposition $p/1 = \gamma\oplus^2\alpha$ forces $\alpha = 1/0$. Since $\F(0)$ and $\F(1)$ are already defined and $M$ is invertible, we can satisfy Equation (\ref{eq:FRFdef}) in the definition of an FRF by defining $\F(p/1)$ using the second order linear recursion matrix $M$. That is, define
\[ \begin{bmatrix}\F(p-1)\\ \F(p)\end{bmatrix}  = M^{p-1}\begin{bmatrix}\F(0)\\ \F(1)\end{bmatrix}\]
for all $p \in \Z$. Indeed, for $p/1 = (p-2)/1\oplus^21/0 \in F_1$, the definition gives $\F(p/1) = -d_1(1/0)\F(p-2) +d_2(1/0)\F(p-1)$ which is Equation (\ref{eq:FRFdef}).

Now, assume for some $n \in \mathbb{N}$, $\F$ is defined on the Farey sequence $F_n$ and if $r/s \in F_n$ with $r/s = \gamma\oplus^2 \alpha$, then Equation (\ref{eq:FRFdef}) holds.  Let $p/q \in F_{n+1}-F_n$ and, using Lemma \ref{unique_triple}, fix the unique $\gamma=x/y$, $\alpha = z/w \in \Q$ satisfying $p/q = \gamma\oplus^2\alpha$.   By Remark \ref{rem:size}, $\gamma$, $\alpha$, and $\gamma \oplus \alpha$ are all elements of $F_n$.  Therefore, $\F(p/q)$ can (and must) be defined by
\[ \F(p/q) = -d_1(\alpha) \F(\gamma) + d_2(\alpha) \F(\gamma \oplus \alpha).\]
By uniqueness of the pair $\gamma, \alpha$, $\F$ is well defined and for all $r/s \in F_{n+1}$ Equation (\ref{eq:FRFdef}) holds. This inductive definition for $\F$ proves the theorem.
\end{proof}

Define $\Q_+=\Q \cap [0,\infty)$ and $\whq_+= \Q_+ \cup \{ 1/0\}$.

\begin{thm}\label{thm:shortQ} If $a$, $b$, and $c$ are elements of a commutative ring $R$ and $d_1$ and $d_2$ are functions from $\whq_+$ to $R$, then there is a unique $(d_1,d_2)$-Farey recursive function $\F\co\whq_+ \to R$ with $\F\left(0/1 \right)=a$,  $\F\left(1/0 \right)=b$, and  $\F\left(1 \right)=c$.
\label{thm:main2}
\end{thm}

\begin{proof} This follows the same proof as Theorem \ref{thm:kenton}, but now there is no need extend to the negative integers, so the condition that $d_1(1/0)$ is invertible is unnecessary.
\end{proof}

\begin{cor} \label{cor:main}
If $a$, $b$, and $c$ are elements of a commutative ring $R$ and $d \co \whq \to R$ is a function for which $d(1/0)$ is invertible in $R$, then there is a unique Farey recursive function $\F\co\whq \to R$ with determinant $d$ with $\F\left(1/0 \right)=a$,  $\F\left(1/0 \right)=b$, and  $\F\left(1/1 \right)=c$.
\end{cor}

\begin{proof} This follows from the proof of Theorem \ref{thm:main}.  Use $d_2=\F$ and follow the proof, noting that for $p/q =\gamma\oplus^2\alpha$, Equation (\ref{eq:FRFdef}) only requires $d_2$ to be evaluated at $\alpha$. For $q\geq 2$, $\alpha \in F_{q-1}$ and for $q=1$, $\alpha =1/0$. Hence $d_2(\alpha) = \F(\alpha)$ is already defined when using induction to define $\F(p/q)$.
\end{proof}

\section{More examples} \label{sec:examples} 

Armed with Theorems \ref{thm:kenton} and \ref{thm:shortQ} and Corollary \ref{cor:main}, we can discuss more examples.  

\begin{exmpl}[Fibonacci numbers] \label{ex:Fib}
Define $F \co \whq \to \Z$ to be the Farey recursive function with constant determinant $d=-1$ and initial conditions $F(0/1)=0$, $F(1/0)=1$ and $F(1/1)=1$.  Since $F(n)$ is the $n^\text{th}$ Fibonacci number, $F$ is a Farey recursive extension of the Fibonacci sequence to the extended rationals.
\end{exmpl}

\begin{exmpl}[Generic FRFs] As defined in \cite{C}, the {\em generic $(d_1,d_2)$-FRF} is $\mathcal U\co\whq \to\Z[x,y,z]$ defined by $d_1$, $d_2$ and the triple $x,y,z$. Every FRF factors through $\mathcal U$. That is, given a $(d_1,d_2)$-FRF $\F\co\whq \to R$, $\F = f\circ \mathcal U$ where $f\co\Z[x,y,z] \to R$ is defined by the ring homomorphism sending $x \mapsto \F(0/1)$, $y \mapsto \F(1/0)$ and $z \mapsto \F(1/1)$.

The {\em generic FRF with determinant $d$} is the FRF $\mathcal{U}_d \co \whq \to Z[x,y,z]$ as defined above with determinant $d$ (i.e., as usual $d_1=d$ and $d_2 = \mathcal U$).
As above, every FRF with image in a ring $R$ and determinant $d$ is a specialization of $\mathcal U_d$.
\end{exmpl}

The following examples have applications in low-dimensional topology.  Some terminology from topology appears below without explanation - sensible definitions would be a substantial diversion.  Interested readers will be directed to other references for more details.

\begin{exmpl}[Traces of matrices] \label{ex:U} The generic FRF with constant determinant 1 is particularly useful because the polynomials in the image of $\mathcal U_1$ can be used to calculate the traces of certain $2 \times 2$ matrices.

Let $G$ be the fundamental group for a once punctured torus with marked generators of slopes $1/0$ and $0/1$.  Suppose that $\rho \co G \to \text{SL}_2\C$ is a group homomorphism.  Take $x$, $y$, and $z$ to be the traces of the matrices $\rho(a)$, $\rho(b)$, and $\rho(c)$ where $a$, $b$, and $c$ are primitive elements of $G$ with respective slopes $0/1$, $1/0$, and $1/1$.  Then, as shown in \cite{C}, the value of the specialization of $\mathcal U_1 (p/q)$ is the trace of $\rho(g)$ where $g$ is a primitive element of $G$ with slope $p/q$.
\end{exmpl}

\begin{exmpl}[Markov numbers] Another specialization of $\mathcal U_1$ is related to the Markov numbers (see \cite{A} for a thorough introduction to Markov numbers).  An integer is called a {\em Markov number} if it is part of an integer solution to the equation
\[ x^2+y^2+z^2 = 3xyz.\]
Let $\mathcal M$ be the FRF with determinant one obtained from $\mathcal U_1$ by setting $x=y=z=3$.  It follows from Section 1 of \cite{Bow} that $\frac 13 \, \mathcal M\left(\Q\cap[0,1]\right)$ is the set of Markov numbers.
\end{exmpl}

The next few examples  are relevant to the study of a class of topological objects called {\em two bridge links} (see \cite{Ad} and \cite{BZH}).  An element $p/q \in \whq$ determines an embedding of either a loop or pair of loops into $\R^3$.   Figure \ref{fig:47knot} shows the embedding for $4/7$.  This comes from drawing the slope $p/q$ arcs emanating from the corners of a square pillowcase in $\R^3$ and connecting the left corners of the pillow with one arc and the right corners with another (as shown in the figure).  This link (or knot) is referred to as the two bridge link $L(p/q)$.   This rich class of links have been long studied by mathematicians.  In \cite{Sch}, Schubert described a connection between the two bridge links, continued fractions, and the Farey graph.  This important relationship was, in particular, used in the famous papers \cite{HT} and \cite{SW} to establish fundamental topological properties of these links.  This set of links continues to provide an important class of examples of spaces in low dimensional topology and geometry.

\begin{figure}[ht] 
\centering
\includegraphics[width =2.5in]{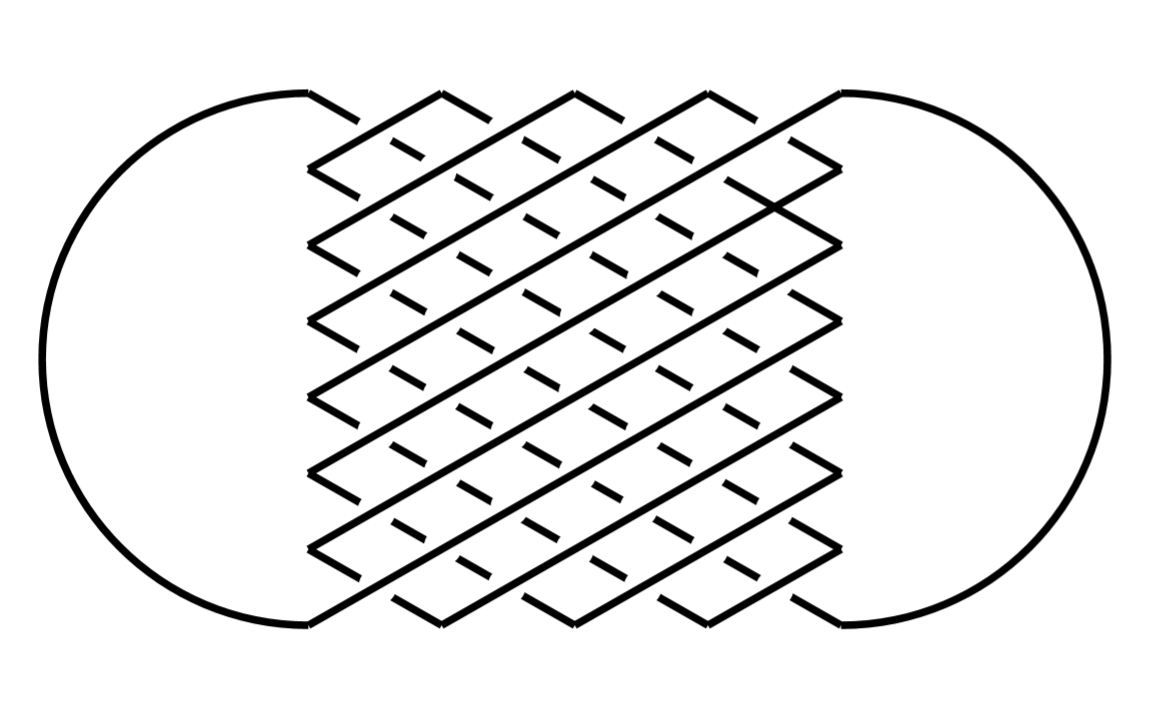}
\caption{The knot $L(4/7)$.}
\label{fig:47knot}
\end{figure}

\begin{exmpl}[Character varieties and Riley polynomials for 2-bridge links] 
From \cite{C}, let $\mathcal T \co \whq \to \Z[x,z]$ be the FRF with determinant one obtained from $\mathcal U_1$ by setting $y=0$.  The main theorems of \cite{C} show that the affine set \[\{ (x_0,z_0) \in \C^2 \, | \, \mathcal T (p/q)(x_0,z_0)=0\}\] corresponds to the set of homomorphisms from the fundamental group for the link $L(p/q)$ into $\text{PSL}_2 \C$.  

This projective matrix group is especially relevant here for geometric reasons.   By some measures, the simplest two bridge links can be drawn on the surface of a torus.  For those who cannot, their complements are examples of {\em hyperbolic manifolds}, 3-dimensional spaces which have natural geometric structures modeled on 3-dimensional hyperbolic space $\mathbb{H}^3$.  In fact, every two bridge link is represented by an element $p/q$ of $\whq_0$ and, by William Thurston's celebrated geometrization theorems (see for example \cite{Kap}, \cite{Mor}, \cite{Otal}, and \cite{Th}) these are {\em hyperbolic} precisely when $p\notin{1,q-1}$.   The isometry group for $\mathbb{H}^3$ can be identified with $\text{PSL}_2 \C$ and this guarantees that, for the hyperbolic two bridge links there is always an isomorphism from the link group to a subgroup of $\text{PSL}_2 \C$.  In fact, this isomorphism can be obtained directly from a root of the one variable {\em  Riley polynomial} for $L(p/q)$ (see \cite{RR} and \cite{RR2}).  From \cite{C}, the Riley polynomial for $L(p/q)$ is a specialization of $\mathcal T(p,q)$.
\end{exmpl}

\begin{exmpl}[Geometry of 2-bridge links]
Suppose that $L(p/q)$ is a hyperbolic two bridge link.  The main result of \cite{SW} constructs a triangulation for the complement of $L(p/q)$ by ideal tetrahedra.  It was proven independently in \cite{PTGTKG} and \cite{GF} that these triangulations always carry the geometric structure for the link complements.  This reduces the problem of explicitly finding the geometry of $L(p/q)$ to solving a complicated system of multivariable polynomial equations.  It is shown in \cite{C2} that there is a FRF with image in $\Z[x]$, where the geometry of $L(p/q)$ comes from a root of the image of $p/q$ under this FRF. 

Let $d_\calQ \co \whq \to \Z[x]$ be the muliplicative FRF from Example \ref{ex:dQ}, $d_\calQ(p/q) = (-1)^px^q$, and define $\mathcal Q \co \whq \to \Z[x]$ be the FRF with determinant $d_Q$ and 
\begin{align*} \calQ(1/0)&=0 & \calQ(1/1)&=1 & \calQ(0/1)& =1. \end{align*}
Then, from \cite{C2}, the geometry of $L(p/q)$ corresponds to a root $\mathfrak{z}$ of $\calQ(p/q)$.  In particular, the geometric shapes of the tetrahedra in the triangulation given by Sakuma and Weeks in \cite{SW} correspond to complex numbers obtained by evaluating quotients of certain values of $\calQ$  at $\mathfrak{z}$.

The polynomials $\calQ(1/n)$ are closely related to the Chebyshev polynomials.  Recall that the Chebyshev polynomials of the second kind $\{ U_n (x)\}$ are defined by the second order linear recurrence relation $U_{n+1}=-U_{n-1}+2xU_n$ and initial conditions $U_0=0$ and $U_1=1$.  The Chebyshev polynomials of the fourth kind are given by $W_n=U_n+U_{n-1}$ and satisfy the same recurrence relation.   Using induction, it is easy to show that
\[ \calQ\left(\frac{1}{2n}\right)=x^{n-1} U_n\left(\frac{1-2x}{2x}\right) \quad \text{and} \quad \calQ\left(\frac{1}{2n-1}\right)=x^{n-1} W_n\left(\frac{1-2x}{2x}\right).\]
\end{exmpl}

\begin{rem}
It seems worth noticing that, in each of our examples, every FRF is a FRF with determinant $d$ where $d$ is a multiplicative FRF as in Example \ref{ex:mult}.
\end{rem}


\section{Wrapping sequences around triangles} \label{sec:wrap} 

For this section, assume that $R$ is a commutative ring and $d \co \whq \to R$ is a multiplicative function as defined in Example \ref{ex:mult}.  Assume also that the image of $d$ contains no zero divisors.  Let $\F \co \whq \to R$ be a Farey recursive function with determinant $d$.  This seems to be a natural setting.  In particular, all of the examples from Section \ref{sec:examples} have these properties.
For $\alpha \in \whq$, define
\[ M_\alpha = \begin{bmatrix} 0 & 1 \\ -d(\alpha) & \F(\alpha) \end{bmatrix}\]
and notice that for any Farey pair $\gamma,\,\alpha \in \whq$, by Equation (\ref{eq:FRFdef}),
\begin{equation}\label{eq:lin_recur}
\begin{bmatrix}\F(\gamma\oplus^n\alpha)\\\F(\gamma\oplus^{n+1}\alpha)\end{bmatrix} = M_\alpha^n\begin{bmatrix}\F(\gamma)\\\F(\gamma\oplus\alpha)\end{bmatrix}
\end{equation}
 for all $n\geq 0$. That is, if we restrict $\F$ to the boundary sequence for $\alpha$ containing $\gamma$, the result is linearly  recursive with recursion matrix $M_\alpha$.

To help motivate what we mean by wrapping around a boundary triangle, first recall that the recursion matrix for the Fibonacci numbers is $M=\left[ \begin{smallmatrix} 0&1\\1&1 \end{smallmatrix} \right]$.    Here, $M$ is invertible and $M^{-1}$ has integer entries.  This means that the sequence of Fibonacci numbers $\{f_n\}$ can be extended to a bi-infinite sequence by using $M^{-1}$ to move in the negative direction. For instance, $f_{-1}$ is defined by the equation
\[ \begin{bmatrix}f_{-1}\\f_{0}\end{bmatrix} = \begin{bmatrix}0&1\\1&1\end{bmatrix}^{-1}\,\begin{bmatrix}f_0\\f_1\end{bmatrix} = \begin{bmatrix}-1&1\\1&0\end{bmatrix}\,\begin{bmatrix}0\\1\end{bmatrix} = \begin{bmatrix}1\\0\end{bmatrix}.\]
The resulting bi-infinite sequence is linearly recursive with recursion matrix $M$.   A portion of this bi-infinite Fibonacci sequence is
\[\ldots,13,\,-8,\,5,\,-3,\,2,\,-1,\,1,\,0,\,1,\,1,\,2,\,3,\,5,\,8,\,13,\ldots\]
To summarize, if $f \co \Z_{\geq 0} \to \Z$ is the function which takes $n$ to the $n^\text{th}$ Fibonacci number, then $f$ has a unique linearly recursive extension to the bi-infinite extension $\Z$ of $\Z_{\geq 0}$.  We have in effect taken the Fibonacci numbers and wrapped them around the boundary ``triangle'' $\partial(1/0)$.

As shown in (\ref{eq:lin_recur}), for $\alpha \in \Q$, the restrictions of $\F$ to the boundary sequences $\partial_L(\alpha)$ and $\partial_R(\alpha)$ are linearly recursive with recursion matrix $M_\alpha$.  It is reasonable to concatenate these two boundary sequences to get a single bi-infinite sequence and to consider an analogy to the situation with the Fibonacci numbers described above.  If the sequence $\partial_L(\alpha)$ is reversed and juxtaposed with $\partial_R(\alpha)$, the set $\Delta(\alpha)=\partial_L(\alpha) \cup \partial_R(\alpha)$ becomes a bi-infinite sequence.  As such, write $\Delta(\alpha) = \{ \beta_j \}$ where $\beta_{-1}=\gamma_L$, $\beta_0=\gamma_R$, and for $k \in \mathbb{N}$,
\[ \beta_{-k} = \gamma_L \oplus^{k-1} \alpha \qquad \beta_k = \gamma_R \oplus^k \alpha\]
where $\gamma_L$ and $\gamma_R$ are the left and right corners for $\alpha$. 

 For example, if $\alpha=1/2$ then $\gamma_L=0$ and $\gamma_R=1$.  Furthermore,
\[ \ldots, \ \beta_{-3}=2/5, \ \beta_{-2} = 1/3, \ \beta_{-1}=0, \ \beta_0=1, \ \beta_1= 2/3, \ \beta_2=3/5, \ \ldots \]
Compare this to the portion of the Stern-Brocot diagram shown in Figure \ref{fig:triangles}.  Notice how the sequence $\{ \beta_j \}$ wraps around the triangle centered at $1/2$.

By (\ref{eq:lin_recur}),
\[ M_\alpha \begin{bmatrix} \beta_{j-1} \\ \beta_j \end{bmatrix} =  \begin{bmatrix} \beta_{j} \\ \beta_{j+1} \end{bmatrix}\]
holds for every $j \in \mathbb{N}$,  regardless of $\alpha \in \Q$.  What are conditions on $\F$ which guarantee that it holds for every $j \in \Z$?  The answer to this is the content of Theorem \ref{thm:recursion_formula} and Corollary \ref{cor:det1}.

The crux of the problem occurs at the corners $\beta_{-1}$ and $\beta_0$ for $\alpha$.  This motivates a comparison of  $M_\alpha^{-2} \left[ \begin{smallmatrix}  \F(\beta_0) \\ \F(\beta_{1}) \end{smallmatrix} \right]$ with the values of $\F$ at $\beta_{-1}$ and $\beta_{-2}$.  To start, notice that $\beta_{-1}$, $\alpha=\beta_{-1}\oplus\beta_0$, and $\beta_1=\beta_{-1}\oplus^2\beta_0$ are consecutive in $\partial_L(\beta_0)$ and that $\beta_0$, $\alpha = \beta_0\oplus\beta_{-1}$, and $\beta_{-2}= \beta_0\oplus^2\beta_{-1}$ are consecutive in $\partial_R(\beta_{-1})$.  Because $\F$ is Farey recursive and $d$ is multiplicative, these expressions give the following formulas
\begin{align}
\F(\beta_1) &= -d(\beta_0)\F(\beta_{-1}) + \F(\alpha) \F(\beta_0)  \label{eq:beta1}\\
\F(\beta_{-2}) &= -d(\beta_{-1}) \F(\beta_0) + \F(\alpha) \F(\beta_{-1}) \label{eq:beta-2} \\
d(\alpha) &= d(\beta_{-1}) d(\beta_0). \label{eq:dmult}
\end{align}
So, using Equation (\ref{eq:beta1}) in the last step, 
\begin{eqnarray}
M_\alpha^{-1} \, \begin{bmatrix} \F(\beta_0) \\ \F(\beta_{1}) \end{bmatrix} &=&  \frac{1}{d(\alpha)} \, \begin{bmatrix} \F(\alpha) & -1 \\ d(\alpha) & 0 \end{bmatrix} \,  \begin{bmatrix} \F(\beta_0) \\ \F(\beta_{1}) \end{bmatrix} \nonumber\\ \nonumber \\
&=&  \frac{1}{d(\alpha)} \, \begin{bmatrix} \F(\alpha) \F(\beta_0)-\F(\beta_1) \\ d(\alpha) \F(\beta_0) \end{bmatrix}\nonumber \\ \nonumber \\
&=& \begin{bmatrix} \frac{d(\beta_0)}{d(\alpha)} \F(\beta_{-1}) \\ \F(\beta_0) \end{bmatrix}. \label{eq:minus1}
\end{eqnarray}
Using Equations (\ref{eq:beta-2}) and (\ref{eq:dmult}),
\begin{eqnarray}
M_\alpha^{-2} \, \begin{bmatrix} \F(\beta_0) \\ \F(\beta_{1}) \end{bmatrix} &=& M_\alpha^{-1} \, \begin{bmatrix} \frac{d(\beta_0)}{d(\alpha)} \F(\beta_{-1}) \\ \F(\beta_0) \end{bmatrix} \nonumber \\ \nonumber \\
&=&\frac{1}{d(\alpha)} \, \begin{bmatrix} \F(\alpha) & -1 \\ d(\alpha) & 0 \end{bmatrix} \, \begin{bmatrix} \frac{d(\beta_0)}{d(\alpha)} \F(\beta_{-1}) \\ \F(\beta_0) \end{bmatrix}\nonumber \\ \nonumber \\
&=& \frac{1}{d(\alpha)} \, \begin{bmatrix} \frac{d(\beta_0)}{d(\alpha)} \F(\alpha) \F(\beta_{-1}) -\F(\beta_0) \\ d(\beta_0) \F(\beta_{-1}). \end{bmatrix} \nonumber\\ \nonumber \\
&=& \begin{bmatrix} \frac{1}{d(\beta_{-1})d(\alpha)} \left( \F(\alpha) \F(\beta_{-1}) - d(\beta_{-1}) \F(\beta_0) \right) \\ \frac{1}{d(\beta_{-1})} \F(\beta_{-1}) \end{bmatrix} \nonumber \\ \nonumber \\
&=& \begin{bmatrix} \frac{1}{d(\beta_{-1})d(\alpha)}  \F(\beta_{-2}) \\ \frac{1}{d(\beta_{-1})} \F(\beta_{-1}) \end{bmatrix} \label{eq:mtxcalc}
\end{eqnarray}
This provides the key step for the proof of the following theorem.

\begin{thm} \label{thm:recursion_formula}
Let $d \co \whq \to R$ be a multiplicative function to a commutative ring $R$ whose image contains no zero divisors.  Suppose $\F$ is a Farey recursive function with determinant $d$.  Given $\alpha \in \Q$, define $M_\alpha$ and $\{ \beta_j\}$ as above.  Then, for all $n \in \Z$,
\[
M_\alpha^n \begin{bmatrix}\F(\beta_0) \\ \F(\beta_1)\end{bmatrix} = \left\{
\begin{array}{ll}
\begin{bmatrix}\F(\beta_n)\\ \F(\beta_{n+1})\end{bmatrix}&n\geq 0\\ \\
\begin{bmatrix}\frac1{d(\beta_{-1})}\F(\beta_{-1}) \\ \F(\beta_0)\end{bmatrix} & n=-1\\ \\
\begin{bmatrix}\frac{1}{d(\beta_{-1})d(\alpha)^{-n-1}}\F(\beta_n)\\  \frac{1}{d(\beta_{-1})d(\alpha)^{-n-2}}\F(\beta_{n+1})\end{bmatrix}&n<-1
\end{array}
\right.
\]
\end{thm}

\begin{proof} For $n\geq 0$ the formula holds for all Farey Recursive Functions and hence holds in our case. For $n=-1$ and $-2$ the formulas follow from Equations (\ref{eq:minus1}) and (\ref{eq:mtxcalc}).  Now, proceed by induction. Assume the formula holds for a fixed $n\leq -2$. Then
\begin{eqnarray*}
M_\alpha^{n-1}\begin{bmatrix}\F(\beta_0)\\\F(\beta_1)\end{bmatrix} &=&M_\alpha^{-1}M_\alpha^n\begin{bmatrix}\F(\beta_0)\\\F(\beta_1)\end{bmatrix}\\ \\
&=&\frac1{d(\alpha)}\begin{bmatrix}\F(\alpha)&-1\\d(\alpha)&0\end{bmatrix}\begin{bmatrix}\frac{1}{d(\beta_{-1})d(\alpha)^{-n-1}}\F(\beta_n)\\ \frac{1}{d(\beta_{-1})d(\alpha)^{-n-2}}\F(\beta_{n+1})\end{bmatrix}\\ \\
&=&\begin{bmatrix}\frac{1}{d(\beta_{-1})d(\alpha)^{-n}}\left[\F(\alpha)\F(\beta_n)-d(\alpha)\F(\beta_{n+1})\right]\\ \frac{1}{d(\beta_{-1})d(\alpha)^{-n-1}}\F(\beta_n)\end{bmatrix}\\ \\
&=&\begin{bmatrix}\frac{1}{d(\beta_{-1})d(\alpha)^{-n}}\F(\beta_{n-1})\\ \frac{1}{d(\beta_{-1})d(\alpha)^{-n-1}}\F(\beta_{n})\end{bmatrix}.\\
\end{eqnarray*}
\end{proof}
Note that, if the roles of the left and right corners are interchanged in the discussions above, a similar result holds for the sequences obtained by reversing the sequences $\Delta(\alpha)$.

Since the constant determinant $d=1$ is multiplicative, the last corollary is immediate.
\begin{cor} \label{cor:det1}
Suppose $\F\co\whq\to R$ is a Farey Recursive Function with constant determinant one.  Given $\alpha \in \Q$, define $M_\alpha$ and $\{ \beta_j \}$ as before.  Then, for all $n \in \Z$, 
\[
M_\alpha^n \begin{bmatrix}\F(\beta_0) \\\F(\beta_1)\end{bmatrix} = 
\begin{bmatrix}\F(\beta_n)\\  \F(\beta_{n+1})\end{bmatrix}.
\]
\end{cor} 

In particular, since the Farey Recursive Functions in Examples \ref{ex:Fib} and \ref{ex:U} have constant determinant 1, the functions wrap both ways around all triangles in the Stern-Brocot Diagram.

\bibliographystyle{plain}
\bibliography{FRF}

\end{document}